\documentclass[a4paper,10pt,english]{amsart}

\usepackage[latin1]{inputenc}

\usepackage{calc}
\usepackage{graphicx,palatino,verbatim}

\usepackage{amsmath}

\usepackage[T1]{fontenc}
\usepackage[english]{babel}

\usepackage{amsthm}

\usepackage{amssymb}

\usepackage{amsmath}

\usepackage{latexsym}

\usepackage{amsfonts}
\usepackage{syntonly}
\usepackage{graphicx}
\usepackage{mathptmx}
\usepackage{subfig}
\usepackage{url}
\usepackage{lscape}
\usepackage[letterpaper]{geometry}
\input xy
\xyoption{all}

\newtheorem{theorem}{Theorem}[section]

\newtheorem{lemma}[theorem]{Lemma}
\newtheorem{definition}[theorem]{Definition}
\newtheorem{example}[theorem]{Example}

\newtheorem{proposition}[theorem]{Proposition}

\newtheorem{cor}[theorem]{Corollary}

\renewcommand{\H}{{H_{\textrm{head}}}}
\newcommand{\T}{{H_{\textrm{tail}}}}

\newcommand     {\Zset}    {{\mathbb Z}}
\DeclareMathOperator{\Maps} {Maps}
\DeclareMathOperator{\op} {op}

\def\quotient#1#2{%
    \raise1ex\hbox{$#1$}\Big/\lower1ex\hbox{$#2$}%
}

\begin{document}

\title
{Hecke-Kiselman monoids of small cardinality}
\author[R.~Aragona]{Riccardo Aragona}
\email{ric\_aragona@yahoo.it}

\author[A.~D'Andrea]{Alessandro D'Andrea}
\email{dandrea@mat.uniroma1.it}

\date{\today}

\maketitle

\begin{abstract}
In this paper, we give a characterization of digraphs $Q, |Q|\leq 4$ such that the associated Hecke-Kiselman monoid $H_{Q}$ is finite.
In general, a necessary condition for $H_Q$ to be a finite monoid is that $Q$ is acyclic and its Coxeter components are Dynkin diagram. We show, by constructing examples, that such conditions are not sufficient.
\end{abstract}
\section{Introduction}
Let $Q$ be a digraph, i.e., a graph having at most one connection (side) between each pair of distinct vertices; sides can be either oriented (arrows) or non-oriented (edges). In \cite{gm}, Ganyushkin and Mazorchuk associate with $Q$ a semigroup $H_Q$ generated by idempotents $a_i$ indexed by vertices of $Q$, subject to the following relations
\begin{itemize}
\item $a_ia_j=a_ja_i$, if $i$ and $j$ are not connected;
\item $a_ia_ja_i=a_ja_ia_j$, if $(\xymatrix@-1pc{i\ar@{-}[r]&j}) \in Q$, i.e., $i$ and $j$ are connected by a side;
\item $a_ia_j=a_ia_ja_i=a_ja_ia_j$, if $(\xymatrix@-1pc{i\ar[r]&j}) \in Q$, i.e., $i$ and $j$ are connected by an arrow from $i$ to $j$.
\end{itemize}
$H_Q$ is the \textit{Hecke-Kiselman monoid} attached to $Q$.

In \cite{f}, Forsberg proves faithfulness of certain representations of Hecke-Kiselman monoids and constructs some classes of such representations. Hecke-Kiselman monoids also appear in the works \cite{g} and \cite{p} of Grensing,  where she studies projection functors $P_S$ attached to simple modules $S$ of a finite dimensional algebra, which satisfy the above defining relations.

The two extremal type of digraphs are graphs, where all sides are edges, and oriented graphs, in which all sides are arrows. When $Q$ is the full graph on $\{1,2,\ldots,n\}$ with the natural order, then the corresponding Hecke-Kiselman monoid is Kiselman's monoid $K_n$ from \cite{k, km}. $K_n$ is known to be finite \cite[Theorem 3]{km} for all $n$. If a digraph $Q$ only has arrows, but possesses no oriented cycles, then $H_Q$ is isomorphic to a quotient of $K_{|Q|}$, hence it is finite.

If a digraph $Q$ has no arrows --- in particular, when it is a finite simply laced Coxeter graph --- then $H_Q$  is the \textit{Springer-Richardson}, \textit{$0$-Hecke}, or \textit{Coxeter monoid} attached to $Q$. Monoid algebras over $0$-Hecke monoids were studied by Norton \cite{n} in the finite-dimensional case: this corresponds to requiring that $Q$ is a Dynkin diagram. Notice that finite Coxeter monoids also appear in the work \cite{rs} of Springer and Richardson on the combinatorics of Schubert subvarieties of flag manifolds. In \cite{dhst} and in \cite{gm}, the results of Norton were interpreted and studied within the framework of $J$-trivial semigroups. In particular, finite $0$-Hecke monoids and Kiselman monoids, along with their quotients,  are examples of $J$-trivial monoids (see \cite[Chapter IV, Section 5]{l}).

The problem of determining finiteness of the Hecke-Kiselman monoid associated to a digraph $Q$ with both edges and arrows, appears to be combinatorially involved and is, to the best of our knowledge, unsettled. In this paper we give some conditions on a digraph $Q$ for $H_{Q}$ to be finite. We produce a complete classification of finite $H_Q$ when $|Q|\leq 4$.

It is easy to see that if $H_Q$ is a finite monoid, then $Q$ is acyclic (see Definition \ref{acyclic}) and the Coxeter graph $C$ obtained from $Q$ by removing all arrows is necessarily a Dynkin diagram (Corollary \ref{dynkin}). The main observation in this paper is that these two properties do not provide a characterization of digraphs of finite type, as the combinatorics of arrows plays a fundamental role in determining the finiteness character of $H_Q$.

\section{Cycles in Hecke-Kiselman monoids}

We will say that a digraph $Q$ is of {\em finite type} whenever the monoid $H_Q$ is finite.
A digraph $Q$ and the digraph $Q^{\op}$, obtained from $Q$ by reversing each arrow, yield anti-isomorphic Hecke-Kinselman monoids. As a consequence, $Q$ is of finite type if and only if $Q^{\op}$ is.
\begin{lemma}\label{remove}
Let $Q$ be a digraph of finite type. If $Q'$ is obtained from $Q$ by orienting an edge, or removing an arrow, then $Q'$ is of finite type.
\end{lemma}
\begin{proof}
It follows from \cite[Proposition 14]{gm}.
\end{proof}
\begin{cor}\label{dynkin}\qquad
\begin{itemize}
\item
Let $C$ be a Dynkin diagram, $Q$ an oriented graph obtained from $C$ by choosing an orientation of every edge. Then $H_{Q}$ is finite.
\item
Let $Q$ be a digraph of finite type, $C$ be the graph obtained from $Q$ by removing every arrow.
Then $C$ is a Dynkin diagram.
\end{itemize}
\end{cor}
\begin{proof}
A simply laced Coxeter graph is of finite type if and only if it is a Dynkin diagram.
\end{proof}
As an example, in the case of Kiselman's monoid $K_n$, removing all arrows yields a disjoint union of $n$ components of type $A_1$. It is important to notice that removing sources or sinks\footnote{Recall that a vertex of a digraph is a {\em source} (resp. a {\em sink}) if all sides touching it are outgoing (resp. incoming) arrows.} from a digraph does not affect its finiteness character.
\begin{proposition}\label{mxl}
Let $Q$ be a digraph. If $a\in Q$ is a source (reps. sink) vertex, then $axa=ax$ (resp. $axa=xa$) for every $x\in H_Q$. In particular, if $Q'$ is obtained from $Q$ by removing $a$ and every edge connected to $a$, then $Q'$ is of finite type if and only if $Q$ is.
\end{proposition}
\begin{proof}
It follows from \cite[Lemma 1]{km}.
\end{proof}


\begin{definition}\label{acyclic}
Let $Q$ be a digraph. A cycle in $Q$ is a sequence $\{a_i, i \in \Zset/n\Zset\}, n\geq 3$, of vertices of $Q$ such that there exists in $Q$ an edge or an arrow going from $a_i$ to $a_{i+1}$.
A cycle only composed of arrows is an {\em oriented cycle}.
We say that a digraph is acyclic if it contains no cycles.
\end{definition}
\begin{example}
The following are both cycles:
$$
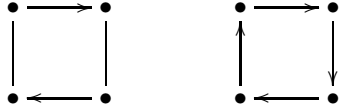
\begin{array}{ccccc}
\xymatrix{\bullet\ar[r]&\bullet\ar@{-}[d]\\
\bullet\ar@{-}[u]&\bullet\ar[l]}
&&&&
\xymatrix{\bullet\ar[r]&\bullet\ar[d]\\
\bullet\ar[u]&\bullet\ar[l]}
\end{array}
$$
However, only the latter is an oriented cycle.
\end{example}

Let $a_1,\ldots, a_n$ be generators of $H_Q$. If $Q$ is an oriented graph, we set $a_i>a_j$ whenever there is an arrow connecting $a_i$ to $a_j$ and we take the transitive closure of this relation.
When $Q$ is acyclic, we obtain a partial ordering on $Q$, that we may always refine to a (non necessarily unique) total order.
\begin{lemma}\label{orientedcycle}
The $n$-cycle
$$
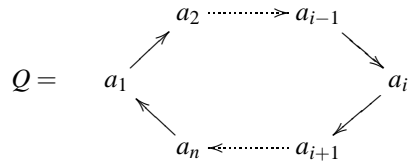
\xymatrix@-1pc{&&a_2\ar@{.>}[rr]&&a_{i-1}\ar[dr]&\\
Q=&a_1\ar[ur]&&&&a_i\ar[dl]\\
&&a_n\ar[ul]&&a_{i+1}\ar@{.>}[ll]&}
$$

is not of finite type.
\end{lemma}
\begin{proof}
The collection $\Maps(\Zset^n)$ of all maps $f: \Zset^n \to \Zset^n$ is a semigroup under composition. Our strategy is to construct a semigroup homomorphism $\rho: H_Q \to \Maps(\Zset^n)$ and show that its image is infinite. Notice that, due to the presentation of $H_Q$, $\rho$ is given as soon as we choose images $u_i = \rho(a_i),\, i = 1, \dots, n$ satisfying the defining relations of $H_Q$.
Let $u_i: \Zset \to \Zset, i = 1, \dots, n$ be defined as follows:
\begin{itemize}
\item $u_i(m_1,\ldots,m_n)=(m_1,\dots,m_{i-1},m_{i+1},m_{i+1},\ldots,m_n)$,\qquad if $i = 1, \dots, n-1$;
\item $u_n(m_1,\ldots,m_n)=(m_1,\ldots,m_{n-1},m_1+1)$.
\end{itemize}
A straightforward check shows that $u_i$ satisfy the
defining relations.
However,
$$
(u_1\dots u_n)(m_1,\ldots,m_n)=(m_1+1,m_1+1,\ldots,m_1+1),
$$
showing that all powers of $u_1 \dots u_n$ are distinct. We conclude that the image of $\rho$ is infinite, hence $H_{Q}$ is too.
\end{proof}
\begin{theorem}\label{noncycles}
A digraph of finite type is acyclic.
\end{theorem}
\begin{proof}
Assume that $Q$ contains a cycle, and denote by $Q'$ the digraph obtained from $Q$ by removing all connections not belonging to the cycle, and orienting the remaining edges so as to form an oriented cycle. Then $Q'$ is of infinite type by Proposition \ref{mxl} and Lemma \ref{orientedcycle}.
\end{proof}

Let $Q$ be a finite
digraph and $Q'$ be obtained from $Q$ by removing all arrows. Then
$Q'$ is a disjoint union of (finitely many) uniquely determined connected graphs, called \textit{Coxeter components of} $Q$. We have already seen in Corollary \ref{dynkin} that if $Q$ is of finite type, then all of its Coxeter components are of Dynkin type. Absence of cycles in a finite digraph imposes geometrical constraints on the arrows.
\begin{proposition}\label{order}
Let $Q$ be an acyclic digraph. Then the set of Coxeter components of $Q$ can be totally ordered in such a way that an arrow connects a vertex in the Coxeter component $C$ to a vertex in the Coxeter component $C'$ only if $C>C'$.
\end{proposition}
\begin{proof}
We are going to show the existence of a Coxeter component $C$ with only outgoing arrows. The statement then follows by setting $C$ to be maximal, and using induction to determine the total order on remaining components.

Assume by contradiction that $Q$ has no maximal Coxeter component. Then, every Coxeter component $C$ of $Q$ has an incoming arrow and we can find $C' \neq C$ such that there is an arrow from $C'$ to $C$. We can thus build a sequence $C_0, C_1, \dots, C_n$ of Coxeter components of arbitrary length, so that there is an arrow from $C_{i+1}$ to $C_i$ for every $i$. Due to finiteness of $Q$, there can only be finitely many Coxeter components. As each Coxeter component is connected, there must exist a cycle in $Q$.
\end{proof}
A total order as above may fail to be unique. For instance, every total order on a totally disconnected digraph satisfies the requirements of Proposition \ref{order}.

\section{Digraphs of small cardinality}
The tools we have developed so far allow one to classify digraphs of finite type of very small cardinality. When addressing digraphs of larger cardinality, we encounter more complicated combinatorial issues. In this section we will be dealing only with acyclic digraphs. Recall that if all Coxeter components of a digraph $Q$ are of type $A_1$, i.e., they are isolated points, then $Q$ is a quotient of a Kiselman monoid, hence it is of finite type.
\begin{theorem}\label{upto3}
Every acyclic digraph of cardinality at most three is of finite type.
\end{theorem}
\begin{proof}
If $Q$ has no arrows, then it is of Dynkin type, and the corresponding monoid is finite. If $Q$ has more than one Coxeter component, then it must have either a sink or a source, whence we may apply an easy induction.
\end{proof}
Let now $Q$ be an acyclic digraph with $|Q|=4$. If $Q$ is not connected, then it is a disjoint union of digraphs of smaller cardinality and it is of finite type by Theorem \ref{upto3}. So, assume $Q$ to be connected.

If $Q$ has no arrows, then $Q$ is of Dynkin type $D_4$ or $A_4$, hence it is of finite type. If $Q$ has at least a Coxeter component of type $A_1$, then it is of finite type by Proposition \ref{mxl} and Theorem \ref{upto3}. Thus, we only need to understand the case where $Q$ has exactly two Coxeter components of type $A_2$. In all that follows, $K$ will denote the digraph
$$
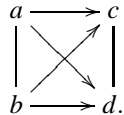
\xymatrix{a\ar[r]\ar[dr]&c\ar@{-}[d]\\
b\ar@{-}[u]\ar[r]\ar[ur]&d.}
$$
Let $\T$, $\H$ denote the submonoids of $H_K$ generated by $\{a, b\}, \{c, d\}$ respectively. Notice that both $\T$ and $\H$ are isomorphic images of $H_{A_2}$, as $H_K$ projects to $H_{A_2}$ by collapsing either $\{a, b\}$ or $\{c, d\}$ to $1$. If $w \in H_K$, let $l(w)$ denote the length of a reduced expression of $w$ as a products of elements $a, b, c, d$.

\begin{lemma}\label{wordK(A2A2)}
Let $w \in H_K$. Then there exist elements
$\{w_i\,|\,1\leq i\leq n\}\subseteq \T$ and $\{v_i\,|\,1\leq i\leq n\}\subseteq \H$ such that
\begin{equation}\label{reduced}
w = w_0 v_n w_1 v_{n-1} \dots v_1 w_n v_0
\end{equation}
where $v_i \neq 1, w_i \neq 1$ for all $i \neq 0$ and
\begin{itemize}
\item[(i)] if $l(w_i)=3$, then either $i = n = 0$ or $i=n=1$ and $w_0 = 1$;
\item[(ii)] if $1 < i < n$, then $l(w_i)=1$;
\item[(iii)] if $l(w_i) = 1$, then $w_{i-1} w_i \neq w_{i-1}$ if $i \neq 0$ and $w_i w_{i+1} \neq w_{i+1}$ if $i \neq n$. In particular, $w_i \neq w_{i+1}$ for $1 < i <n$;
\item[(iv)] if $l(w_i) = 2$, and $i \neq 0, n$, then $i=1$ and $w_0 = 1$. Moreover, if $l(w_i) = l(w_{i+1}) = 2$, then $w_i = w_{i+1}$;
\end{itemize}
and similarly,
\begin{itemize}
\item[(i)] if $l(v_i)=3$, then either $i = n = 0$ or $i=n=1$ and $v_0 = 1$;
\item[(ii)] if $1 < i < n$, then $l(v_i)=1$;
\item[(iii)] if $l(v_i) = 1$, then $v_i v_{i-1} \neq v_{i-1}$ if $i \neq 0$ and $v_{i+1}v_i \neq v_{i+1}$ if $i \neq n$. In particular, $v_i \neq v_{i+1}$ for $1 < i <n$;
\item[(iv)]  if $l(v_i) = 2$, and $i \neq 0, n$, then $i=1$ and $v_0 = 1$. Moreover, if $l(v_i) = l(v_{i+1}) = 2$, then $v_i = v_{i+1}$.
\end{itemize}
\end{lemma}
\begin{proof}
We will henceforth assume that the product of all nontrivial terms in \eqref{reduced} is a reduced expression for $w$ in terms of $a, b, c, d$.
We first prove that $w_i$ satisfy properties \textit{(i)}-\textit{(iv)}.

First of all, observe that we may assume that if $w_i = 1$ for some $i \neq 0$, then we may drop it, and multiply the two adjacent terms.
\begin{itemize}
\item[(i)]  If $l(w_i)=3$ then $w_i=aba=bab$. By Proposition \ref{mxl}, we may remove all occurrences of $a$ and $b$ appearing on the right of $w_i$, and conclude that $i=n$.
Still by Proposition \ref{mxl}, $x \, a \, v \, aba = x \, a \, v \, ba$ and $x \, b \, v \, bab = x \, b\, v \, ab$ can be further simplified for every $v \in \H$. By the reducedness assumption, one has $n = 0$ or $n = 1$ and $w_0 = 1$.

\item[(ii)] By property (i), we note that $l(w_i)>1$, with $2 \leq i \leq n-1$, implies $w_i\in\{ab,ba\}$. Say that $w_i = ab$ for some $i \geq 2$. We want to show that $i = n$. Indeed, $w_{i-1} \neq 1$ by the initial observation, and $w_{i-1} \neq a, ba, aba$ otherwise $w$ may be further simplified by replacing $w_i = ab$ with $w_i = b$. Then, $w_{i-1}$ equals either $b$ or $ab$.  However, in this case, $w_{i-1} v_{n-i + 1} w_i = w_{i-1} v_{n-i+1} bab$ and, as before, we may cancel all $w_j, j > i$. Reducedness of \eqref{reduced} then implies $i = n$. The case $w_i = ba$ is totally analogous.

\item[(iii)] Use again Proposition \ref{mxl}. If $w_{i-1}w_i = w_{i-1}$, then canceling $w_i$ gives an expression for $w$ of lower length. If $w_i w_{i+1} = w_{i+1}$, then $w_{i+1}$ begins by $w_i$, and one may reduce $w$ to a shorter expression.

\item[(iv)] Assume that $i \neq 0$ and $w_{i-1} \neq 1$. If $w_{i-1}$ has length one, then $w_{i-1} w_i \neq w_i$ has necessarily length $3$; it is easy to check that this also happens if $w_{i-1}$ has higher length. Then one may replace $w_{i-1} v_{n-i+1} w_{i}$ with $w_{i-1} v_{n-i+1} aba$ in $w$ and cancel all $w_j, j > i$. This show that either $i = n$ or $w_{i-1} = 1$, which is only possible if $i = 1$.

As for the last statement, notice that $ab\, v\, ba = ab\, v \,a$ and $ba\, v \,ab = ba\, v\, b$ by Proposition \ref{mxl}, hence we may assume $w_i = w_{i+1}$ by the reducedness assumption.
\end{itemize}
The proof for the $v_i$ is totally analogous.
\end{proof}

In simple words, Lemma \ref{wordK(A2A2)} says that $w_0$ is the only possibly trivial element among the $w_i$. Moreover, if $aba=bab$ appears among the $w_i$, then it is the only nontrivial one, and terms of length two only show up at the beginning and the end of \eqref{reduced}; if they are followed (resp. preceded) by a term of length one, they do not end (resp. begin) by that term; two adjacent terms of length two are necessarily equal. All remaining $w_i$ are of length one, and no two adjacent ones are equal, so as to avoid possible simplifications. The same description applies to the $v_i$.

\begin{cor}\label{corwordK(a2a2)}
Every element in $H_{K}$ can be expressed as
\begin{equation}\label{word}
w(xyzt)^nw',
\end{equation}
where $n\in\mathbb{N}$, $\{x, z\} = \{a, b\}$, $\{y, t\} = \{c, d\}$, and $l(w),l(w') \leq 10$.
\end{cor}
\begin{proof}
We can certainly group $4n$ adjacent $w_i, v_j$ of length one, so that they are preceded (resp. followed) by at most two $w_i$ of length one along with a non trivial $w_i$ of different length, and similarly for the $v_i$. Then the product of the $4n$ terms is a power of $xyzt$ as in the statement, and terms preceding and following it have length at most $2(1 + 1 + 3) = 10$.
\end{proof}

\begin{cor}\label{4finite}
A quotient of $K$ is finite if and only if $acbd$ and $adbc$ have finitely many distinct powers.
\end{cor}
\begin{proof}
Follows immediately from Corollary \ref{corwordK(a2a2)}.
\end{proof}

\begin{lemma}\label{K(A2A2)}
$K$ is not of finite type.
\end{lemma}
\begin{proof}
We define an action of the generators $a$, $b$, $c$ and $d$ of $H_{K}$ on the set of the vertices $V$ of the infinite graph in Figure \ref{fig1}.

Each generator act on any given vertex according to the arrow originating from the vertex with the corresponding label, with the understanding that the generator fixes the vertex if there is no outgoing arrow with that label.
A straightforward check shows that $a$, $b$, $c$ and $d$ have idempotent actions of $V$, and they furthermore satisfy the defining relations:
\begin{itemize}
\item $aba=bab$;
\item $cdc=dcd$;
\item $ac=aca=cac$;
\item $ad=ada=dad$;
\item $bc=bcb=cbc$; and
\item $bd=bdb=dbd$.
\end{itemize}

We conclude that $K$ is of infinite type, as distinct powers of $acbd$ (resp. $adbc$) have distinct actions on the central vertex $A_0$ (resp. the vertex $B_0$).
\end{proof}

\begin{landscape}
\begin{figure}[!h]
$$
\vspace{30mm}\xymatrix@-1.60pc{&&&&&&&&&&&&&&&&&A_1\ar@{=}[d]&&&&&&&&&&&&&&&&&\\
&&&&&&&&&C_0\ar[d]
^c\ar[dl]_a&&&&B_0\ar[dd]
^c\ar[llll]_b&&&&A_0\ar[dd]
^c\ar[llll]_a\ar[rrrr]^b&&&&B_{1}\ar[dd]
^c\ar[rrrr]^a&&&&C_{1}\ar[d]
^c\ar[dr]^b&&&&&&&&&\\
&&&&&&&&X&F_0\ar[l]^a\ar[d]^b\ar[rr]^d&&E_0\ar[r]^a\ar[dddd]^b|!{[d]}\hole&X&&&&&&&&&&X&E_1\ar[l]_b\ar[dddd]^a|!{[d]}\hole&&F_1\ar[r]_b\ar[d]_a\ar[ll]_d&X&&&&&&&&\\
&&&&&C_3\ar[d]^d\ar[dl]_b&&&&B_3\ar[llll]_a\ar[dd]^d&&&&A_3\ar[llll]_b\ar[dd]^d\ar[drr]^a&&&&L_0\ar[dll]_a\ar[drr]^b&&&&A_2\ar[rrrr]^a\ar[dd]^d\ar[dll]_b&&&&B_2\ar[rrrr]^b\ar[dd]^d&&&&C_2\ar[d]^d\ar[dr]^a&&&&&\\
&&&&X&F_3\ar[l]^b\ar[rr]^c\ar[d]^a&&E_3\ar[r]^b\ar[dddd]^a|!{[d]}\hole&X&&&&&&&P_0\ar[dd]^b&&L_1\ar@{=}[u]&&P_1\ar[dd]^a&&&&&&&X&E_2\ar[l]_a\ar[dddd]^b|!{[d]}\hole&&F_2\ar[r]_a\ar[ll]_c\ar[d]_b&X&&&&\\
&C_4\ar[d]^c\ar[dl]_a&&&&B_4\ar[dd]^c\ar[llll]_b&&&&A_4\ar[llll]_a\ar[drr]^b\ar[dd]^c&&&&L_3\ar[dll]_b\ar[urr]^a&&&&&&&&L_2\ar[ull]_b\ar[drr]^a&&&&A_5\ar[rrrr]^b\ar[dll]_a\ar[dd]^c&&&&B_5\ar[dd]^c\ar[rrrr]^a&&&&C_5\ar[d]^c\ar[dr]^b&\\
X&F_4\ar[l]^a\ar[rr]^d\ar[d]^b&&E_4\ar[r]^a\ar[dddd]^b|!{[d]}\hole&X&&&&&&&P_3\ar[dd]^a&&&&Q_0\ar[dd]^a&&&&Q_1\ar[dd]^b&&&&P_2\ar[dd]^b&&&&&&&X&E_5\ar[l]_b\ar[dddd]^a|!{[d]}\hole&&F_5\ar[r]_b\ar[ll]_d\ar[d]_a&X\\
&B_7\ar[dd]^d\ar@{.>}[l]&&&&A_7\ar[llll]_b\ar[drr]^a\ar[dd]^d&&&&L_4\ar[dll]_a\ar[urr]^b&&&&&&&&&&&&&&&&L_5\ar[ull]_a\ar[drr]^b&&&&A_6\ar[rrrr]^a\ar[dll]_b\ar[dd]^d&&&&B_6\ar[dd]^d\ar@{.>}[r]&\\
&&&&&&&P_4\ar[dd]^b&&&&Q_3\ar[dd]^b&&&&X&&&&X&&&&Q_2\ar[dd]^a&&&&P_5\ar[dd]^a&&&&&&&\\
&A_8\ar@{.>}[l]\ar[drr]^b\ar@{.>}[dd]&&&&L_7\ar[dll]_b\ar[urr]^a&&&&&&&&&&&&&&&&&&&&&&&&L_6\ar[ull]_b\ar[drr]^a&&&&A_9\ar@{.>}[r]\ar[dll]_a\ar@{.>}[dd]&\\
&&&P_7\ar[dd]^a&&&&Q_4\ar[dd]^a&&&&X&&&&&&&&&&&&X&&&&Q_5\ar[dd]^b&&&&P_6\ar[dd]^b&&&\\
&&{}\ar@{.>}[ur]&&&&&&&&&&&&&&&&&&&&&&&&&&&&&&{}\ar@{.>}[ul]&&\\
&&&Q_7\ar[dd]^b&&&&X&&&&&&&&&&&&&&&&&&&&X&&&&Q_6\ar[dd]^a&&&\\
&&&&&&&&&&&&&&&&&&&&&&&&&&&&&&&&&&\\
&&&X&&&&&&&&&&&&&&&&&&&&&&&&&&&&X&&&}
$$
\caption{}\label{fig1}
\end{figure}
\end{landscape}

\begin{theorem}
$K$ is the only acyclic digraph of infinite type with four vertices.
\end{theorem}
\begin{proof}
We only need to handle the case when the digraph $Q \neq K$ has two Coxeter components of type $A_2$. By Lemma \ref{K(A2A2)}, it suffices to prove that the digraph
$$
Q':
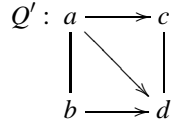
\xymatrix{a\ar[r]\ar[dr]&c\ar@{-}[d]\\
b\ar@{-}[u]\ar[r]&d}
$$
is of finite type, since $H_Q$ is a quotient of $H_{Q'}$.

Let us compute all powers of $x = adbc$. Notice that $adbc = adcb$ as $b$ and $c$ commute, and that $ay = aya$ (resp. $by = byb$) for every $y \in \langle c, d\rangle$ as $ad = ada, ac = aca$ (resp. $bd = bdb, bc = bcb$). Then, $x = adcb = adcab$, hence
$$
x^2=(adcab)(adcb)=adc(aba)dcb=adc(bab)dcb = adcba(bdcb) = adcba(bdc) = adc(bab)dc,
$$
and
\begin{equation*}\begin{split}
x^3 =x^2 x = &adc(aba)dc(adbc)= adcab(adca)dbc= adcab(adc)dbc = adc(aba)dcdbc = \\& adc(bab)dcdbc = adcba(bdcdb)c = adcba(bdcd)c = adc(bab)(dcdc) = adc(aba)(cdc).
\end{split}
\end{equation*}
However $dc(aba)cdc = (aba)cdc$ as $czc = zc, dzd = zd$ for all $z \in \langle a, b\rangle$. It is now easy to check $a, b, c, d$ act trivially by right multiplication on $x^3 = aba\, cdc$, hence $x^n = x^3$ for all $n > 3$. Thus $x$ has only finitely many distinct powers.

A similar proof works for $acbd$, and we conclude that $Q'$ is of finite type by using Corollary \ref{4finite}.
\end{proof}
It is likely that our techniques may be extended to handle the case of two Coxeter components of any Dynkin type. However, characterizing the combinatorics of all digraphs of finite type with three or more Coxeter components appears to be much more difficult.



\vfill

\end{document}